\documentclass[a4paper,10.1pt]{amsart}

\usepackage{amsmath}
\usepackage{amssymb}
\usepackage{amsthm}
\usepackage{verbatim}
\usepackage[all]{xy}
\usepackage{enumerate}

\usepackage{hyperref}

\usepackage{color}

\usepackage{float}

\usepackage{algorithm}
\usepackage[noend]{algpseudocode}

\usepackage[all]{xy}

\newtheorem{thm}{Theorem}[section]
\newtheorem{prop}[thm]{Proposition}
\newtheorem{cor}[thm]{Corollary}
\newtheorem{lem}[thm]{Lemma}

\theoremstyle{definition}

\newtheorem{rmk}[thm]{Remark}

\newtheorem{conj}[thm]{Conjecture}

\numberwithin{equation}{section}
\newcommand{\cM}{\mathcal{M}}

\newcommand{\CB}{\mathcal{CB}}

\newcommand{\id}{\textrm{id}}

\newcommand{\cB}{\mathcal{B}}

\newcommand{\supp}{{\rm supp}}

\newcommand{\cS}{\mathcal{S}}

\newcommand{\CBS}{\mathcal{M}^{cb}}
\newcommand{\BS}{\mathcal{M}}

\newcommand{\B}{\mathcal{B}}
\newcommand{\ran}{{\rm ran}}

\MakeRobust{\Call}

\title[On the complete bounds of $L_p$-Schur multipliers]{On the complete bounds of $L_p$-Schur multipliers}

\date{\today, {\it MSC2010}: 47B10, 47L20, 47A30. Key words: Schur multipliers, Non-commutative $L_p$-spaces, Operator spaces.}

\address{M. Caspers, G. Wildschut, TU Delft, DIAM, Analysis, Van Mourik Broekmanweg 6, 2628 XE Delft, The Netherlands}

\author{Martijn Caspers and Guillermo Wildschut}

\begin{document}

\maketitle

\begin{abstract} 
We study the class $\BS_p$ of Schur multipliers on the Schatten-von Neumann  class $\cS_p$ with $1 \leq p \leq \infty$   as well as the class of completely bounded Schur multipliers $\CBS_p$. 
We first show that for $2 \leq p < q \leq \infty$ there exist  $m \in  \CBS_p$ with $m \not \in \BS_q$,  so in particular the following inclusions that follow from interpolation are strict  $\BS_q \subsetneq  \BS_p$ and $\CBS_q \subsetneq \CBS_p$. In the remainder of the paper we collect computational evidence that for $p\not = 1,2, \infty$ we have $\BS_p = \CBS_p$, moreover with equality of bounds and complete bounds. This would suggest that a conjecture raised in \cite{PisierAsterisque} is false.
\end{abstract}

 \section{Introduction}
 
The Schur product of matrices is given by the entry-wise product. For $m \in M_n(\mathbb{C})$ the linear map 
\[
M_m: M_n(\mathbb{C}) \rightarrow M_n(\mathbb{C}): x \mapsto M_m(x) := (m_{i,j} x_{i,j})_{i,j},
\]
is then called a Schur multiplier.  

 Schur multipliers appear in several different contexts.  They are widely applied in harmonic analysis because of their close connection with Fourier multipliers and transference techniques, see e.g. \cite{BozejkoFendler},  
  \cite{NeuwirthRicard},  \cite{CaspersDelaSalle}. In operator theory Schur multipliers of divided differences occur naturally in problems involving commutators of operators, see e.g.   \cite{PotapovSukochevActa} and references given there. Further, recently new applications of transference techniques have been found in approximation properties of Lie groups \cite{LafforgueDelaSalle}, \cite{Laat}, \cite{LaatSalle}. In each of these applications crucial new results were obtained on the (complete) bounds of Schur multipliers. 

 The boundedness properties of $M_m$ depend on the norm imposed on $M_n(\mathbb{C})$. If $M_n(\mathbb{C})$ is equipped with the operator norm, the bounds of $M_m$ can be described by Grothendieck's characterization \cite[Theorem 5.1]{PisierBook}. In particular the bounds and complete bounds of a Schur multiplier agree and in an infinite dimensional setting we see that every bounded multiplier is in particular automatically completely bounded.  
 If $M_n(\mathbb{C})$ is equipped with the Schatten $\cS_p$-norm then finding bounds, or even optimal bounds, of $M_m$ becomes very complicated as there is no such charaterization as Grothendieck's available.

In the current paper we show two things. Let $\cM_p$ (resp. $\cM_p^{cb}$) be the collection of symbols $m$ that are bounded (resp. completely bounded) Schur multipliers of  the Schatten-von Neumann classes $\cS_p$ associated with an infinite dimensional Hilbert space. We refer to Section \ref{Sect=Prelim} for exact definitions. 
  Through complex interpolation we have that $\cM_q \subseteq \cM_p$ in case $2 \leq p < q \leq \infty$. We show that this inclusion is strict; in fact we get a slightly stronger result in particular yielding the parallel result on the complete bounds as well. This extends the results by Harcharras \cite[Theorem 5.1]{Harcharras} which proves this for even $p$ and it settles the question of strict inclusions (the problem was also stated in \cite[p. 51]{Doust}). 
  
  Secondly, we study the question whether $\cM_p$ and $\cM_p^{cb}$ are equal  for $1 < p \not = 2 < \infty$. In fact, the following conjecture is stated  in \cite{PisierAsterisque}: 
  
  \begin{conj}[Conjecture 8.1.12 in \cite{PisierAsterisque}]\label{Conj=Pisier} 
  	For every $1 < p \not = 2 < \infty$ we have that $\BS_p \not = \CBS_p$. 
  \end{conj}
  
  If we replace $\BS_p$ and $\CBS_p$ by the class of respectively the bounded and completely bounded {\it Fourier} multipliers on a locally compact abelian group then  Conjecture \ref{Conj=Pisier} is true as is proven in \cite{PisierAsterisque} in case of the torus and \cite{Arhancet} for arbitrary locally compact abelian groups. Pisier's argument relies on lacunary sets in $\mathbb{Z}$ (for the bounds) and transference to Schur multipliers and unconditionality of the matrix units as a basis for $\cS_p$ (for failure of the complete bounds). From this perspective it is very reasonable to state Conjecture \ref{Conj=Pisier}.
  
  In \cite{LafforgueDelaSalle}  it was proved that  for continuous Schur multipliers on $\cB(L_2(\mathbb{R}))$ we have $\cM_p = \cM_p^{cb}$ with {\it equal} bounds and complete bounds as operators on $\cS_p(L_2(\mathbb{R}))$. The continuity is essential in their proof and leaves open Conjecture \ref{Conj=Pisier}. It deserves to be noted that Lafforgue and De la Salle find several other   fundamental properties of Schur multipliers in the same paper \cite{LafforgueDelaSalle}.

   In the current paper we approximate the norms of Schur multipliers by computer algorithms; they suggest that $\cM_p  = \cM_p^{cb}$ with equality of norms and complete norms (just as in the case $p=\infty$). We  show that this is true in case of the triangular truncation (Corollary \ref{Cor=Trunc}).   

\vspace{0.3cm}

\noindent {\it Acknowledgements.} Authors thank C\'edric Arhancet and the anonymous referee for useful comments on the contents of this paper.

 \section{Preliminaries}\label{Sect=Prelim}
 
 \subsection{Schatten classes $\cS_p$} Let $H$ be a Hilbert space and let $\cB(H)$ be the space of bounded operators on $H$. For $1 \leq p < \infty$ we let $\cS_p = \cS_p(H)$ be the space of  operators $x \in \cB(H)$ such that 
 \[
 \Vert x \Vert_p := {\rm Tr}(\vert x \vert^p)^{\frac{1}{p}} < \infty.
 \]
 The assignment $\Vert \: \Vert_p$ defines a norm on $\cS_p$ which turns it into a Banach space and which is moreover an ideal in $\cB(H)$. We set $\cS_\infty$ for the C$^\ast$-algebra of compact operators with operator norm $\Vert  \: \Vert_\infty$. In case $H = \mathbb{C}^n$ we write $\cS^n_p$ for $\cS_p = \cS_p(\mathbb{C}^n)$. Fixing an orthonormal basis $f_j, j \in \mathbb{N}_{\geq 1}$ we have that we may identify  $\cS_p^n$  (completely) isometrically  as a subspace of $\cS_p$ by mapping the matrix unit $e_{i,j} \in \cS_p^n$ to the matrix unit $e_{f_{i}, f_j} \in \cS_p$ given by $e_{f_{i}, f_j} f_k = \langle f_k, f_j \rangle f_i$.  Let $P_n$ be the projection onto the span of $f_1, \ldots, f_n$. Then this map is an isometric isomorphism $\cS_p^n \simeq P_n \cS_p P_n$. Moreover, under this isomorphsim $\cup_n \cS_p^n$ is dense in $\cS_p$. In case $1 \leq p \leq q \leq \infty$ we have $\cS_p \subseteq \cS_q$ and the inclusion is (completely) contractive. This in particular turns  $(\cS_p, \cS_q)$ into a compatible couple of Banach spaces and for any $p \leq  r \leq q$ we have that $\cS_r$ is a complex interpolation space between $(\cS_p, \cS_q)$, see \cite{BerghLofstrom}, \cite{PisierOS}. Any tensor product $\cS_p^n \otimes \cS_p$ will be understood as a $L_p$-tensor product, i.e. the $p$-norm closure as a subspace of $\cS_p(\mathbb{C}^n \otimes H)$.

 \subsection{Operator space structure} For the theory of operator spaces we refer to \cite{EffrosRuan}, \cite{PisierOS}; we shall only need a result of Pisier on completely bounded maps on Schatten classes which we recall here. In \cite{PisierAsterisque} Pisier shows that $\cS_p$ have a natural operator space structure as interpolation spaces between $\cS_1$ and $\cS_\infty$. In \cite{PisierAsterisque} it was proved that a linear map $M: \cS_p \rightarrow \cS_p$ is completely bounded iff for every $s \in \mathbb{N}$ the amplification
 \[
 \id_s \otimes M: \cS_p^s \otimes \cS_p \rightarrow \cS_p^s \otimes \cS_p
 \]
 is bounded with bound uniform in $s$. Moreover,
 \begin{equation}\label{Eqn=CompleteBound}
 \Vert M: \cS_p \rightarrow \cS_p \Vert_{\CB(\cS_p)} = \sup_{s \in \mathbb{N}}  \Vert (\id_s \otimes M): \cS_p^s \otimes \cS_p \rightarrow  \cS_p^s \otimes \cS_p \Vert_{\B(\cS_p^s \otimes \cS_p)}. 
 \end{equation}
 The reader may take \eqref{Eqn=CompleteBound} as a definition, other properties (besides interpolation) of the operator space structure of $\cS_p$ shall not be used in this text.

 \subsection{Schur multipliers} A symbol is a function $m: \mathbb{Z} \times \mathbb{Z} \rightarrow \mathbb{C}$. We call $m$ an $L_p$-Schur multiplier if there exists a  map $M_m: \cS_p \rightarrow \cS_p$ determined by,
 \[
 M_m: \cS_p^n \rightarrow \cS_p^n: (x_{i,j})_{i,j} \mapsto  (m(i,j) x_{i,j})_{i,j}, 
 \]
 here we view again $\cS_p^n$ as a subspace of $\cS_p$ by fixing a basis. From the closed graph theorem, as $M_m$ is presumed to be defined on all of $\cS_p$, the map $M_m$ is automatically bounded. The   space of all $L_p$-Fourier multipliers will be denoted by $\BS_p$ which carries the operator norm $\Vert  \cdot \Vert_{\BS_p}$ of $\cB(\cS_p)$. This turns $\cM_p$ into a Banach space. We denote $\CBS_p$ for the subset of $m \in \BS_p$ such that $M_m: \cS_p \rightarrow \cS_p$ is completely bounded. We equip  $\CBS_p$ with the completely bounded norm $\Vert \cdot \Vert_{\CBS_p}$ as completely bounded maps on $\cS_p$. With slight abuse of terminology we shall refer to both the symbol $m$ as well as the map $M_m$ as a Schur multiplier and usually write $\Vert M_m \Vert_{\BS_p}$ for $\Vert m \Vert_{\BS_p}$ (and similarly for the completely bounded norms).   Obviously $\CBS_p \subseteq \BS_p$. The question whether this inclusion is strict remains open, see Conjecture \ref{Conj=Pisier}.


 \section{Strict inclusions of the set of Schur multipliers}
 
  Here we prove that for   $2 \leq p < q \leq \infty$ there exists a symbol $m: \mathbb{Z} \times \mathbb{Z} \rightarrow \mathbb{C}$ that is a completely bounded $L_p$-Schur multiplier but which fails to be a bounded $L_q$-Schur multiplier.  
 The following lemma is based on  \cite[Lemma 1]{CowlingFournier}.  
 For a finitely supported meausre $\mu$ on the torus $\mathbb{T}$ we write $\mu \ast$ for the convolution operator $L_p(\mathbb{T}) \rightarrow L_p(\mathbb{T}): f \mapsto \mu \ast f$. We let $\Vert \mu \Vert$ be the norm of the measure. 
 
 \begin{lem}\label{Lem=Cowling}
 	Let $2 \leq p \leq \infty$. There exists a finitely supported measure $\mu_n, n \in \mathbb{N}$  on the torus $\mathbb{T}$ such that,
 	\begin{equation}\label{Eqn=MuEstimate}
 		2^{n/p} \leq  \Vert \mu_n \ast   \Vert_{\B( L_p(\mathbb{T}) )} \leq    \Vert  \mu_n \ast   \Vert_{\CB( L_p(\mathbb{T}) )} \leq \sqrt{2} \: 2^{n/p}.
 		\end{equation}
 \end{lem}
 \begin{proof}
 	Let $s_n = e^{\frac{ \pi i }{2^n}}$.   	
 	Set $\mu_0 = \nu_0 = \delta_1$, the Dirac delta measure in $1 \in \mathbb{T} \subseteq \mathbb{C}$. Then define inductively,
 	\[
 	\mu_{n+1} = \mu_n +   s_n \ast \nu_n, \qquad 	\nu_{n+1} = \mu_n -   s_n \ast \nu_n,
 	\]
 	and note that the supports of  $\mu_n$  and   $s_n \ast \nu_n$ have empty intersection. 
 	
 	We claim that for  every $f \in C_c(\mathbb{T}, \cS_2^m), m \in \mathbb{N}$, we have,
 	\begin{equation}\label{Eqn=SquareNorm}
         \Vert \nu_n \ast f \Vert_2^2 +    \Vert \mu_n \ast f \Vert_2^2 = 2^{n+1} \Vert f \Vert_2^2. 
 	\end{equation}
 	Indeed, this is clear for $n = 0$ and further by the parallellogram law,
 	\[
 	   \Vert \nu_{n+1} \ast f \Vert_2^2 +    \Vert \mu_{n+1} \ast f \Vert_2^2 = 2 ( \Vert \nu_n \ast f \Vert_2^2 +    \Vert s_n \ast \mu_n \ast f \Vert_2^2  )
 	   = 2 ( \Vert \nu_n \ast f \Vert_2^2 +    \Vert  \mu_n \ast f \Vert_2^2  ).
 	\]
 	Then \eqref{Eqn=SquareNorm} follows by induction. From \eqref{Eqn=SquareNorm} we obtain that,
 	\[
 	\Vert \mu_n \ast f \Vert_2^2 \leq 	\Vert \mu_n \ast f \Vert_2^2 + 	\Vert \nu_n \ast f \Vert_2^2 = 2^{n+1} \Vert f \Vert_2^2.
 	\]
 	So that $\Vert \mu_n \ast \Vert_{\CB(L_2(\mathbb{T}))} \leq 2^{(n+1)/2}$.  	Also $\Vert \mu_n \ast \Vert_{\CB(L_1(\mathbb{T}))} \leq \Vert \mu_n \Vert = 2^n$ and by duality also  $\Vert \mu_n \ast \Vert_{\CB(L_\infty(\mathbb{T}))} \leq 2^n$. By complex interpolation therefore $\Vert \mu_n \ast \Vert_{\CB(L_p(\mathbb{T}))} \leq  \sqrt{2} \: 2^{n/p}$. This proves the upperbound in \eqref{Eqn=MuEstimate}.
 	
 	For the lower bounds let $f \in C_c(\mathbb{T})$ be a function with small support close to $1 \in \mathbb{T}$. If the support is small enough, then $\mu_n \ast f$ consists of $2^n$ disjointly supported translates of $f$ so that $\Vert \mu_n \ast f \Vert_p = 2^{n/p} \Vert f \Vert_p$. This yields the lower bound. 
 \end{proof}

 The following theorem shows in particular that the class of $L_p$-Schur multipliers  can be distinguished from the  $L_q$-Schur multipliers     for $2 \leq p < q \leq \infty$. 
 
 \begin{thm}\label{Thm=NonInclusion}
 	Let $2 \leq p < q \leq \infty$ or $1 \leq q < p \leq 2$. There exists symbol $m: \mathbb{Z}^2 \rightarrow \mathbb{C}$ such that $m \in \CBS_p$ but $m \not \in \BS_q$.  
 \end{thm}
\begin{proof}
	We first treat the case $2 \leq p < q < \infty$. 	
   Let $\mu_n$ be the finitely supported measure on $\mathbb{T}$ of Lemma \ref{Lem=Cowling}. Let $m_n: \mathbb{Z} \rightarrow \mathbb{C}$ be its Fourier transform given by
   \[
      m_n(k) = \sum_{\theta \in \supp(\mu_n)} \mu_n(\theta) e^{ik \theta}. 
   \]
   Then set $\widetilde{m}_n: \mathbb{Z}^2 \rightarrow \mathbb{C}$ by $\widetilde{m}_n(k,l) = m_n(k - l)$. By \cite[Theorem 1.2]{NeuwirthRicard} or \cite[Theorem 4.2 and Corollary 5.3]{CaspersDelaSalle} we have
   \begin{equation}\label{Eqn=SchurEq1}
   \Vert \mu_n \ast \Vert_{\CB(L_p(\mathbb{T}))} =    \Vert M_{\widetilde{m}_n} \Vert_{\CB( \cS_p  )}.
   \end{equation}
 	We amplify $\widetilde{m}_n$ by defining $\widetilde{m}_n^{cb}: \mathbb{Z}^2 \times \mathbb{Z}^2 \rightarrow \mathbb{C}: (k,l) = (k_1, k_2, l_1, l_2) \mapsto \widetilde{m}(k_1, l_1)$. Then,
 	\begin{equation}\label{Eqn=SchurEq2}
 	 \Vert M_{\widetilde{m}_n^{cb}} \Vert_{\B( \cS_p  )} = \Vert M_{\widetilde{m}_n^{cb}} \Vert_{\CB( \cS_p  )} =
 	 \Vert M_{\widetilde{m}_n} \Vert_{\CB( \cS_p  )}.
 	\end{equation}
 	Combining \eqref{Eqn=SchurEq1} and \eqref{Eqn=SchurEq2} with the estimates obtained in Lemma \ref{Lem=Cowling} we find that,
 	\begin{equation}\label{Eqn=SchurMainEstimates}
 	    2^{n/p} \leq  \Vert M_{\widetilde{m}_n^{cb}} \Vert_{\B( \cS_p  )}, \qquad {\rm and} \qquad \Vert M_{\widetilde{m}_n^{cb}} \Vert_{\CB( \cS_q  )} \leq \sqrt{2} \: 2^{n/q}.
 	\end{equation}
 	From these two estimates we are able to prove the theorem as follows. 
 	
 	Suppose that the theorem is false, so that we have an inclusion map  $i: \CBS_p \rightarrow \BS_q$. By the closed graph theorem this inclusion is continuous. Indeed, if $k_j \in \CBS_p$ is a net in symbols such that $k_j \rightarrow 0$ in $\CBS_p$ and such that $k_j$ converges to $k$ in $\BS_q$. Then for every matrix $x \in \cS_p^n \subseteq \cS_p$ we find that  $M_{k_j}(x) \rightarrow 0$ in $\cS_p^n$ and hence also in the norm of $\cS_q^n$.  This shows that for  $x \in \cS_q^n \subseteq \cS_q$  we have $M_k(x)   = \lim_j   M_{k_j}(x) = 0$.  But by density of $\cup_n \cS_q^n$ in $\cS_q$ we get that  $M_k(x) = 0$. Hence the graph of $i$ is closed indeed. 
 	
 	However, the estimates \eqref{Eqn=SchurMainEstimates} show that, 
 	\[
 	   \frac{\Vert M_{\widetilde{m}_n^{cb}} \Vert_{\B( \cS_p  )} }{  \Vert M_{\widetilde{m}_n^{cb}} \Vert_{\CB( \cS_q  )} } \geq 2^{ \frac{n}{p} - \frac{n}{q} - \frac{1}{2}},
 	\]
 	which converges to infinity if $n \rightarrow \infty$. This contradicts that $i: \CBS_p \rightarrow \BS_q$ is bounded.

 	Now, if  $1 <  q < p \leq 2$ then the statement follows from duality as  
 	$M_{m}^\ast = M_{m^\vee}$,   	where $m^\vee(k,l) = \overline{m(k,l)}$ and duality preserves the (complete) bounds of linear maps.  	   In case  $q = 1$ or $q = \infty$ the counter example is given by triangular truncation, see \cite{DDPS}. 
\end{proof}

In particular we get the weaker statements that give non-inclusions of bounded and completely bounded multipliers. 

   \begin{cor}
 	Let either $2 \leq p < q \leq \infty$ or $1 \leq q < p \leq 2$. We have that $\BS_p \subsetneq \BS_q$ and $\CBS_p \subsetneq \CBS_q$. 
   	\end{cor}

We may in fact improve on this theorem in the following way.

\begin{cor} \label{Cor=BoundedInterval} 
  Let $2 \leq p < \infty$. There exists a symbol $m \in \CBS_p$ such that for any $q > p$ we have have $m \not \in \BS_q$.  
\end{cor}   	
  \begin{proof}
  	Let $q_n > p$ be a decreasing sequence with $q_n \searrow p$. Let $m_n \in \CBS_p$ with $\Vert m_n \Vert_{\CBS_p} = 1$ be such that $m_n \not \in \BS_{q_n}$. We copy-paste part of the symbols $m_n$ to diagonal blocks of a new symbol $m$ as follows. Let $k_n \in \mathbb{N}$ be such that there exist $x_n \in \cS^{k_n}_{q_n}$ with $\Vert x_n \Vert_q =1$ and  $\Vert M_{m_n} (x_n) \Vert_{q_n} > n$. Let $m_n': [-k_n, k_n] \times [-k_n, k_n] \rightarrow \mathbb{C}$ be the restriction of $m_n$ to a discrete interval. As $\Vert M_{m_n'}(x_n) \Vert_q \geq n$ we see that $\Vert M_{m_{n}'} \Vert \geq n$.   	
  	Then let $m: \mathbb{Z} \times \mathbb{Z} \rightarrow \mathbb{C}$ be the block symbol given by 
  	\[
m = \left(
\begin{array}{ccccc} 
 0 & 0 & \hdots &\hdots&\hdots\\
 0 	 & m_1' & 0& 0& \hdots\\
\vdots   	 & 0 & m_2'  &0& \hdots \\
\vdots  	 & 0 & 0 & m_3'  & \hdots  \\
 \vdots  	 & \vdots & \vdots & \vdots & \ddots \\  	       
\end{array}  
\right).
\] 
  	We find that $M_m = M_{m_1} \oplus  M_{m_2} \oplus M_{m_2} \oplus \ldots$. So that $\Vert M_m \Vert_{\CBS_p} = \sup_k \Vert M_{m_k} \Vert_{\CBS_p} \leq 1$. And similarly, 
  	\[
  	\Vert M_m \Vert_{\CBS_q} = \sup_k \Vert M_{m_k} \Vert_{\CBS_q} \geq \sup_{k, q_k \leq q} \Vert M_{m_k} \Vert_{\CBS_q} = \infty.
  	\]  	
\end{proof}

\begin{rmk}
The proof of Corollary \ref{Cor=BoundedInterval} gives in fact a stronger result. It shows that for any $p_0 > 2$  there is a symbol $m$ such that $M_m: \cS_p \rightarrow \cS_p$ is completely contractive if $p \in [2, p_0]$ and unbounded if $p \in (p_0, \infty)$. 

\end{rmk}

   	\section{Reduction of the variables}\label{Sect=Reduction}

   	Let $n, s \in \mathbb{N}$ and consider $\cS_p^n$. Let $e_{i,i}$ be the diagional matrix unites of $M_n(\mathbb{C})$ and consider the subgroup of $M_{s}(\mathbb{C}) \otimes M_n(\mathbb{C})$ given by all diagonal unitaries $U_1 \otimes e_{1,1} + \ldots + U_{n} \otimes e_{n,n}$ with $U_i \in U(s) \subseteq M_s(\mathbb{C})$ the unitary group. We denote this group by $\oplus_{i=1}^n U(s)$. Naturally $\oplus_{i=1}^n U(s)$ acts isometrically on  $\cS_p^s \otimes \cS_p^n$ by left and right multiplications.    	 
   	
   	\begin{prop}\label{Prop=Polar}
   		Let $m \in \BS_p^n$ and let $s \in \mathbb{N}$. Consider the set of maximum points $C_m^s$ consisting of all  $x \in \cS_p^s \otimes \cS_p^n$ for which $\Vert x \Vert_p = 1$  and such that  $\Vert (\id_s \otimes M_m)(x) \Vert_p =  \Vert \id_s \otimes M_m \Vert_{\BS_p}$.  Then $C_m^s$ is invariant for the left and right action of $\oplus_{i=1}^n U(s)$. In particular, it follows that there exists an $x \in C_m^s$ such that for every $1 \leq i \leq n$ we have that $x_{i,i} := (\id_{s} \otimes \langle \: \cdot \: \: e_i, e_i \rangle)(x) \in \cS_p^s$   is a diagonal matrix with non-negative eigenvalues.  
   	\end{prop}
   	\begin{proof}
   		The first statement is a consequence of the fact that the  Schur multiplier $(\id_s \otimes M_m)$ commutes with the isometric action of $\oplus_{i=1}^n U(s)$. Therefore,  for $U \in \oplus_{i=1}^n U(s)$ we have  $\Vert U x \Vert_p = \Vert x \Vert_p$ and  $\Vert (\id_s \otimes M_m)( U x) \Vert_p = \Vert U (\id_s \otimes M_m)(   x) \Vert_p = \Vert  (\id_s \otimes M_m)(   x) \Vert_p = \Vert \id_s \otimes M_m \Vert_{\cM_p}$.
   		
   		The second statement follows from the polar decomposition. Indeed, take $x \in C_m^s$. We claim first that we may assume that $x_{i,i}$ is a positive semi-definite matrix.  For each $1 \leq i \leq n$ consider the polar decomposition $x_{i,i} = v_i \vert x_{i,i} \vert$ where $v_i$ is a partial isometry with $\ker(v_i)^\perp = \ran( \vert x_{i,i} \vert)$. By dimension considerations we may extend $v_i$ to a unitary $u_i \in U(s)$ that agrees with $v_i$ on $\ran( \vert x_{i,i} \vert)$ so that still $x_{i,i} = u_i \vert x_{i,i} \vert$. Then put $u = \oplus_{i=1}^n u_i \in  \oplus_{i=1}^n U(s)$. Then $u^\ast x \in C_m^s$ by the previous paragraph and $(u^\ast x)_{i,i} = \vert x_{i,i} \vert$ is positive semi-definite. 
   		Let $w_i \in U(s)$ be such that $w_i \vert x_{i,i} \vert w_i^\ast$ is a diaginal matrix, say $d_i$,  with entries $\geq 0$. Then put $w = \oplus_{i=1}^n w_i \in \oplus_{i =1}^n U(s)$. We find that $w^\ast u^\ast x w \in C_m^s$ and further $(w^\ast u^\ast x w)_{i,i}  = d_i$. 
   	\end{proof}
   	
   	Proposition \ref{Prop=Polar}  can be used to significatly speed up our computations in Section \ref{Sect=CPU}. 
   	
   	   	As a side remark we obtain the following corollary that shows that the bounds and complete bounds of an infinite dimensional triangular truncation agree. This result was already recorded in (the discussion before) \cite[Proposition 6.3]{NeuwirthRicard}. In Corollary \ref{Cor=Trunc} we have that  $\Vert h \Vert_{\BS_p}$ is   finite by the stronger result from \cite[Theorem 1.4]{DDPS}.

   	\begin{cor}\label{Cor=Trunc}
   		Let $h: \mathbb{Z}^2 \rightarrow \mathbb{C}$ be the symbol of triangular truncation given by $h(i,j) = \delta_{\geq 0}(i-j)$. Then for every $1 < p < \infty$ we have $\Vert h \Vert_{\BS_p} =  \Vert h \Vert_{\CBS_p}$. 
   	\end{cor}
   	\begin{proof}
   		We use the notation of Proposition \ref{Prop=Polar}. 
   		Let $\pi_s: \cS_p^s \otimes \cS_p \rightarrow \cS_p: e_{i,j} \otimes e_{k,l} \mapsto e_{sk+i, sl+j}$ be the isometric isomorphism that re-indexes matrix units.    	
   		Let $x \in C_h^s, s \in \mathbb{N}_{\geq 2}$. By Proposition \ref{Prop=Polar} we may assume that each $x_{i,i} = (\id_{s} \otimes \langle \: \cdot \: \: e_i, e_i \rangle)(x) \in \cS_p^s, i \in \mathbb{Z}$ is a diagonal matrix. Then $(\id_s \otimes M_h)(x) = M_h(\pi_s(x))$ and therefore, 
   		\[
   		\Vert (\id_s \otimes M_h)(x) \Vert_p =\Vert M_h(\pi_s(x)) \Vert_p \leq \Vert M_h \Vert_{\BS_p} \Vert \pi_s(x) \Vert_p = \Vert M_h \Vert_{\BS_p} \Vert x \Vert_p. 
   		\]   	
   	\end{proof}

\section{Approximation}

     In this section we argue that if Conjecture \ref{Conj=Pisier} is true then we should be able to find computer based evidence for it, which we make precise in the following way. Consider the following three statements:
  \begin{enumerate}
  	\item[1.] For every $1 < p \not = 2 < \infty$ there exists a bounded Schur multiplier that is not completely bounded. 
  	\item[2.] For every $1 < p \not = 2 < \infty$ there exists a completely bounded Schur multiplier $m \in \CBS_p$  such that $\Vert M_m \Vert_{\CBS_p} \not = \Vert M_m \Vert_{\BS_p}$. 
  	\item[3.] For every $1 < p \not = 2 < \infty$ and every  $m \in \BS_p$ we have $m \in \CBS_p$ and moreover $\Vert M_m \Vert_{\CBS_p}  = \Vert M_m \Vert_{\BS_p}$.   
  \end{enumerate}
Statement 1 is Pisier's Conjecture \ref{Conj=Pisier}. 2 is weaker than 1 and 3 is just the negative of 2. If 2 is already true then it is possible to show this by sampling dense sets of Schur multipliers on finite dimensional Schatten classes and by approximating their norms with finite sets. The problem however is that it is not clear how much  computations and computational power is needed in order to obtain a symbol $m$ that witnesses statement 2 above. We state a quantitative statement in this direction in the next proposition.

\begin{prop}\label{Prop=Approx}
	Let $2 \leq p < \infty$.
Fix $n \in \mathbb{N}$ and let $\varepsilon > 0, \delta> 0$. Let $A_\varepsilon$ be the set of all symbols $m: \{1, \ldots, n\}^2 \rightarrow \varepsilon \mathbb{Z}  \cap [-1,1]$ and let $A$ be the  set of all symbols $m: \{1, \ldots, n\}^2 \rightarrow [-1,1]$. Let $B_\delta \subseteq \cS^n_p$ be the set of all $x \in \cS^n_p$ with $\Re (x_{i,j}) \in  \delta  \mathbb{Z} \cap [-1,1]$ and $\Im  ( x_{i,j}) \in  \delta \mathbb{Z} \cap [-1,1]$. For any symbol  $m \in A$   we have that for $\delta < (\sqrt{2} n^{1 + \frac{1}{p}})^{-1}$, 
\begin{equation}\label{Eqn=Delta}
  \Delta_m := \Vert M_m \Vert_{\BS_p}  - \sup_{y \in B_\delta}  \frac{\Vert M_m(y) \Vert_p}{\Vert y \Vert_p}  \leq   \Vert M_m  \Vert_{\BS_p} \frac{ 2 \delta \sqrt{2} n^{1 + \frac{1}{p} }}{ 1- \delta \sqrt{2} n^{1 + \frac{1}{p}} }.
\end{equation}
Further, for every $m \in A_{\varepsilon}$ we have, 
\begin{equation}\label{Eqn=MultiDelta}
\inf_{m' \in A_\varepsilon}  \Vert M_m - M_{m'} \Vert \leq   n \varepsilon .
\end{equation}
\end{prop}
  \begin{proof}
  	Take $x \in \cS_p^n$ be such that $\Vert x \Vert_p = 1$ and $\Vert M_m \Vert_{\BS_p} = \Vert M_m(x) \Vert_p$. Let $x^\delta \in B_\delta$ be such that for each coefficient at entry $i,j$ we have $\vert x_{i,j} -  x_{i,j}^\delta \vert <   \delta  \sqrt{2}$. Let $x_i$ and $x^\delta_i$ be the $i$-th off-diagional of $x$ and of $x^\delta$ respectively. That is, $x_i(k,l) = x(k,l)$ if $k-l = i \mod n$ and $x_i(k,l) = 0$ otherwise. By the triangle inequality, 
  	\[ 
  	\Vert x - x^{\delta} \Vert_p \leq  \sum_{i=1}^n \Vert x_i - x^{\delta}_i \Vert_p  \leq   n ( \delta  \sqrt{2} n^{\frac{1}{p}}). 
  	\]
  	Further  $\Vert x^\delta \Vert_p \leq 1 + \Vert x - x^\delta \Vert_p \leq 1 +  \delta  \sqrt{2} n^{1+ \frac{1}{p}}$ and similarly $\Vert x^\delta \Vert_p \geq 1 - \delta  \sqrt{2} n^{1+ \frac{1}{p}}$. We have, 
  	\[
  	\Vert M_m(x^\delta) \Vert_p \geq \Vert M_m(x) \Vert_p - \Vert M_m(x^\delta - x)\Vert_p \geq \Vert M_m \Vert_{\BS_p} - \Vert M_m \Vert_{\BS_p} \Vert x - x^\delta \Vert_p.
  	\]
   So combining these estimates yields, 
  	\[
  	\begin{split}
  	& \Vert M_m \Vert_{\BS_p} - \frac{\Vert M_m(x^\delta) \Vert_p }{ \Vert x^\delta \Vert_p}
  	=  \frac{ \Vert M_m \Vert_{\BS_p} \Vert x^\delta \Vert_p - \Vert M_m(x^\delta) \Vert_p  }{\Vert x^\delta \Vert_p}  \\
  	\leq &   \frac{ \Vert M_m \Vert_{\BS_p} \Vert x^\delta \Vert_p - \Vert M_m  \Vert_{\BS_p} + \Vert M_m  \Vert_{\BS_p} \Vert x - x^\delta \Vert_p    }{\Vert x^\delta \Vert_p}   \\
  	\leq &  \Vert M_m  \Vert_{\BS_p} \frac{1 + \delta \sqrt{2} n^{1 + \frac{1}{p} }   - 1 + \delta \sqrt{2} n^{1 + \frac{1}{p} }}{ 1- \delta \sqrt{2} n^{1 + \frac{1}{p}} } =   \Vert M_m  \Vert_{\BS_p} \frac{ 2 \delta \sqrt{2} n^{1 + \frac{1}{p} }}{ 1- \delta \sqrt{2} n^{1 + \frac{1}{p}} }. 
  	\end{split}
  	\]
  This proves \eqref{Eqn=Delta}. For \eqref{Eqn=MultiDelta} take $m \in A$. Let $m^\varepsilon \in A_\varepsilon$ be a symbol such that  for each coefficient at entry $i,j$ we have $\vert m_{i,j} -  m_{i,j}^\varepsilon \vert \leq \varepsilon$. Let $m^\varepsilon_i$ be the $i$-th off-diagional of the symbol $m^\varepsilon$;  that is, $m^\varepsilon_i(k,l) = m^\varepsilon(k,l)$ if $k-l = i \mod n$ and $m^\varepsilon_i(k,l) = 0$ otherwise.   Similarly, let $m_i$ be the $i$-th off-diagonal of $m$. 
  We find that,
  \[
  \Vert M_{m} - M_{m^\varepsilon} \Vert_{\BS_p} \leq \sum_{i=1}^n   \Vert M_{m_i} - M_{m^\varepsilon_i} \Vert_{\BS_p} \leq n \varepsilon. 
  \]
 \end{proof}

Proposition \ref{Prop=Approx} shows that we can  approximate the norms of Schur multipliers on $\cS_p^n$. Naturally also the norms of each of the individual matrix amplifications $\id_s \otimes M_{m}, s \in \mathbb{N}$ can be approximated by viewing them as Schur multipliers on 
   	$\cS_p^{sn}$. Note that \label{Eqn=MultiDelta} shows that we may limit ourselves to Schur multipliers taking values in discrete intervals, i.e. with symbol in $A_\varepsilon$. 
   	If Statement 2 above would be true then by approximation we would be able to find counter examples for every $1 < p \not = 2 < \infty$. However, our computer simulations exhibit the behaviour of the converse statement 3.

   \subsection{Approximation with gradient descent methods}  \label{Sect=CPU}
    We have used the  Broyden-Fletcher-Goldfarb-Shanno algorithm (BFGS algorithm, see \cite{Avriel}),  which is a gradient descent algorithm to find local minima/maxima of a function. We apply it here to find local maxima of, 
   \[
   f_m(x) = \frac{\Vert M_m(x) \Vert_p^p}{\Vert x \Vert_p^p} = \frac{ {\rm Tr}\left( (M_m(x)^\ast M_m(x))^{p/2}\right) }{  {\rm Tr} \left( (x^\ast x)^{p/2} \right) }.
   \]
   In case $p \in 2 \mathbb{N}_{\geq 1}$, so that the $p/2$-powers are integer powers, this expression is faster to compute as it avoids  determining eigenvalues of $x^\ast x$ and $M_m(x)^\ast M_m(x)$. 
   The precise algorithm is available on \cite{Algorithm} and it makes use of the reductions in Section \ref{Sect=Reduction}. Note that the sample sets $A_\varepsilon$ in Proposition \ref{Prop=Approx} scale exponentially with the dimension and therefore we are bound to use faster algorithms that only allow us to compute local maxima.    

  \subsection{Approximation for a fixed Schur multiplier}\label{Sect=SingleSchur} In order to illustrate our larger computations we start with the example (fixed) Schur multiplier, 
  \begin{equation} \label{Eqn=ExapleM}
     m = \left(
     \begin{array}{ccc}
       1 & 2 & 2 \\
       -2 & 1 & 3\\
     0 & 2 & -2
     \end{array}
     \right).
  \end{equation}
  The following table shows the approximation of the norm of, 
  \[
  M_m^{(s)} := \id_{\cS_p^s} \otimes M_m = M_{m^{(s)}}, \quad \textrm{    with  symbol } \quad m^{(s)}(i,j) = \left(  \left[  \frac{i}{s}  \right], \left[ \frac{j}{s} \right] \right),
  \]
   where $[r]$ is the largest integer  $k$ with $k \leq r$.

   \begin{figure}[H]\label{Fig=ExampleM}
   	\[
   	\begin{array}{r|r}
   	s		&   \textrm{Approximation of }  \Vert M_{m}^{(s)}\Vert_{\cB( \cS_4^{3s}) }  
   	  \\ \hline \hline 
   	1   & 3.0491549804518234 \\
   	2  & 3.0491549802194102  \\
   	3  & 3.0491549798442240  \\
   	4  & 3.0491549798208277  \\
   	5  & 3.0491549794012864    
   	\end{array}
   	\] 
   	\caption{Approximations of the symbol $m$ of \eqref{Eqn=ExapleM}.}
   \end{figure}

\subsection{Approximations for random Schur multipliers.}\label{Sect=RandomSchur} 
Next we sample random symbols $m$ of Schur multipliers. In the next table $N$ is the number of random samples $m \in M_n(\mathbb{C})$ and to each of these we approximate its norm in essentially the same way as we did to the single example of Section \ref{Sect=SingleSchur}. We then take the maximum over all samples $m$ over the difference of the norms.

\begin{figure}[H]\label{Fig=ExampleM}
	\[
	\begin{array}{r|r|r|r}
	n		&   N
	&  \max_{m} (\Vert M_m^{(2)} \Vert_{\cB(\cS^{2n}_4)} - \Vert M_m \Vert_{\cB(\cS^{n}_{4}  )}) & \max_m( \Vert M_m^{(3)} \Vert_{\cB(\cS^{3n}_{4}  )}) - \Vert M_m \Vert_{\cB(\cS^{n}_{4}  )}))  \\ \hline \hline
	2 & 500 & 0.000000000000000 & -8.881784197001252 \cdot 10^{-16} \\  
	3 & 100 & -3.841578721797134 \cdot 10^{-9} & -3.735681430860893 \cdot 10^{-9} \\ 
	4 & 20 & -2.227329432002989 \cdot 10^{-12} & -5.260509805538049 \cdot 10^{-10} \\  
	\end{array}
	\] 
	\caption{$n$ = dimension of the symbol, $N$ = number of random sample multipliers $m$,  second and third column = approximation of the 2nd and 3th amplification of $m$.}
\end{figure}

Note that the values in this table are negative because it is harder to approximate $\Vert M_m^{(2)} \Vert_{\cB(\cS^{2n}_4)}$ than  $\Vert M_m \Vert_{\cB(\cS^{n}_{4}  )}$. Therefore if it would be true that  $\Vert M_m^{(2)} \Vert_{\cB(\cS^{2n}_4)} = \Vert M_m \Vert_{\cB(\cS^{n}_{4}  )}$ then the approximation of $\Vert M_m^{(2)} \Vert_{\cB(\cS^{2n}_4)}$ is smaller than  the approximation of $\Vert M_m \Vert_{\cB(\cS^{n}_{4}  )}$.

\subsection{Different values of $2 \leq p < \infty$.} For arbitrary $p$ we may still approximate the norm by the same algorithm except that  $\Vert x \Vert_p$ is computed by determining the eigenvalues $\lambda_1, \ldots, \lambda_n$ of $x^\ast x$ so that then $\Vert x \Vert_p^p = \sum_{i=1}^n \vert \lambda_i\vert^{p/2}$. Though that this is computationally more involved we can still carry out our approximations which are displayed in the following figure.

\begin{figure}[H]\label{Fig=ExampleM}
	\[
	\begin{array}{r|r|r|r}
	n		&   N 	&  p = 3 &  p = 3.5   \\ \hline \hline  
	2 & 100 & -1.965094753586527  \cdot 10^{-13}& -2.632338791386246  \cdot 10^{-13}   \\ 
	3 & 10 & -2.616529215515584 \cdot 10^{-10} & -2.0181634141636096 \cdot 10^{-10}\\ 
	\end{array}
	\] 
	
	\[
\begin{array}{r|r|r|r}
n		&   N
&  p = 4.5 &  p = 5   \\ \hline \hline  
2 & 100 & -1.9551027463649007 \cdot 10^{-13} & -3.397282455352979 \cdot 10^{-13} \\ 
3 & 10 &  -3.2847524700230224 \cdot 10^{-10}& -5.515231604746873 \cdot 10^{-9} \\ 
\end{array}
\] 	
	
	\caption{$n$ = dimension of the symbol, $N$ = number of random sample multipliers $m$,  remaining columns   = approximation  over the random sample set of $\max_{m} (\Vert M_m^{(3)} \Vert_p - \Vert M_m \Vert_p)$.}
\end{figure}

\end{document}